\documentclass[12pt]{article}
\font\smallit=cmti10

\usepackage{amssymb,latexsym,amsmath,epsfig,amsthm}
\usepackage{rotating}
\usepackage{graphicx}
\usepackage{amssymb}
\usepackage{lineno}
\usepackage{enumitem}
\usepackage{cite}
\usepackage[top=3cm, bottom=3cm, left=3cm, right=3cm]{geometry}
\makeatletter
\renewcommand\section{\@startsection {section}{1}{\z@}
	{-30pt \@plus -1ex \@minus -.2ex}
	{2.3ex \@plus.2ex}
	{\normalfont\normalsize\bfseries\boldmath}}

\renewcommand\subsection{\@startsection{subsection}{2}{\z@}
	{-3.25ex\@plus -1ex \@minus -.2ex}
	{1.5ex \@plus .2ex}
	{\normalfont\normalsize\bfseries\boldmath}}
\renewcommand{\@seccntformat}[1]{\csname the#1\endcsname. }
\makeatother
\newtheorem{theorem}{Theorem}

\newtheorem{corollary}{Corollary}
\theoremstyle{definition}

\newtheorem{remark}{Remark}

\begin{document}
	
	\begin{center}
		\uppercase{On the minimum size of subset and subsequence sums in integers} 
		\vskip 20pt
		\textbf{Jagannath Bhanja}\\
		{\smallit The Institute of Mathematical Sciences, A CI of Homi Bhabha National Institute, C.I.T. Campus, Taramani, Chennai-600113, India}\\
		{\tt jbhanja@imsc.res.in}\\
		
		\vskip 20pt
		
		\textbf{Ram Krishna Pandey}\\
		{\smallit Department of Mathematics, Indian Institute of Technology Roorkee, Roorkee-247667, India}\\
		{\tt ram.pandey@ma.iitr.ac.in}
	\end{center}
	\vskip 20pt

	\begin{abstract}
		Let $\mathcal{A}$ be a sequence of $rk$ terms which is made up of $k$ distinct integers each appearing exactly $r$ times in $\mathcal{A}$. The sum of all terms of a subsequence of $\mathcal{A}$ is called a subsequence sum of $\mathcal{A}$. For a nonnegative integer $\alpha \leq rk$, let $\Sigma_{\alpha} (\mathcal{A})$ be the set of all subsequence sums of $\mathcal{A}$ that correspond to the subsequences of length $\alpha$ or more. When $r=1$, we call the subsequence sums as subset sums and we write $\Sigma_{\alpha} (A)$ for $\Sigma_{\alpha} (\mathcal{A})$. In this article, using some simple  combinatorial arguments, we establish optimal lower bounds for the size of $\Sigma_{\alpha} (A)$ and $\Sigma_{\alpha} (\mathcal{A})$. As special cases, we also obtain some already known results in this study.
	\end{abstract}
	
	\vspace{10pt}
	
	\noindent {\bf 2020 Mathematics Subject Classification:} 11B75, 11B13, 11B30
	
	\vspace{10pt}
	\noindent {\bf Keywords:} subset sum, subsequence sum, sumset, restricted sumset

\section{Introduction}
Let $A$ be a set of $k$ integers. The sum of all elements of a subset of $A$ is called a \emph{subset sum} of $A$. So, the subset sum of the empty set is $0$. For a nonnegative integer $\alpha \leq k$, let
\[
\Sigma_{\alpha} (A):=\left\{\sum_{a \in A^{\prime}} a: A^{\prime} \subset A,~|A^{\prime}|\geq \alpha\right\} 
\]
and
\[
\Sigma^{\alpha} (A):=\left\{\sum_{a \in A^{\prime}} a: A^{\prime} \subset A,~|A^{\prime}|\leq k-\alpha\right\}. 
\]
That is, $\Sigma_{\alpha} (A)$ is the set of subset sums corresponding to the subsets of $A$ that are of the size at least $\alpha$ and $\Sigma^{\alpha} (A)$ is the set of subset sums corresponding to the subsets of $A$ that are of the size at most $k-\alpha$. So, $\Sigma_{\alpha} (A) = \sum_{a\in A}a - \Sigma^{\alpha} (A)$. Therefore $|\Sigma_{\alpha} (A)| = |\Sigma^{\alpha} (A)|$.

Now, we extend the above definitions for sequences of integers. Before we go for extension, we mention some notation that are used throughout the paper. 

Let $\mathcal{A}=(\underbrace{a_{1},\ldots,a_{1}}_{r~\text{copies}},
\underbrace{a_{2},\ldots,a_{2}}_{r~\text{copies}},\ldots,
\underbrace{a_{k},\ldots,a_{k}}_{r~\text{copies}})$ be a sequence of $rk$ terms, where $a_{1}, a_{2}, \ldots, a_{k}$ are distinct integers each appearing exactly $r$ times in $\mathcal{A}$. We denote this sequence by $\mathcal{A}=(a_{1},a_{2},\ldots,a_{k})_{r}$. If $\mathcal{A'}$ is a subsequence of $\mathcal{A}$, then we write $\mathcal{A'} \subset \mathcal{A}$. By $x \in \mathcal{A}$, we mean $x$ is a term in $\mathcal{A}$. For the number of terms in a sequence $\mathcal{A}$, we use the notation $|\mathcal{A}|$. For an integer $x$, we let $x*\mathcal{A}$ be the sequence which is obtained from by multiplying each term of $\mathcal{A}$ by $x$. For two nonempty sequences $\mathcal{A}$, $\mathcal{B}$, by $\mathcal{A} \cap \mathcal{B}$, we mean the sequence of all those terms that are in both $\mathcal{A}$ and $\mathcal{B}$. Furthermore, for integers $a$, $b$ with $b \geq a$, by $[a,b]_{r}$, we mean the sequence $(a, a+1, \ldots, b)_{r}$.

Let $\mathcal{A}=(a_{1},a_{2},\ldots,a_{k})_{r}$ be a sequence of integers with $rk$ terms. The sum of all terms of a subsequence of $\mathcal{A}$ is called a \emph{subsequence sum} of $\mathcal{A}$. For a nonnegative integer $\alpha \leq rk$, let
\[
\Sigma_{\alpha} (\mathcal{A}):=\left\{\sum_{a\in \mathcal{A^{\prime}}} a: \mathcal{A^{\prime}}\subset \mathcal{A},~|\mathcal{A^{\prime}}|\geq \alpha \right\} 
\]
and
\[
\Sigma^{\alpha} (\mathcal{A}):=\left\{\sum_{a\in \mathcal{A^{\prime}}} a: \mathcal{A^{\prime}}\subset \mathcal{A},~|\mathcal{A^{\prime}}| \leq rk-\alpha \right\}. 
\]
That is, $\Sigma_{\alpha} (\mathcal{A})$ is the set of subsequence sums corresponding to the subsequences of $\mathcal{A}$ that are of the size at least $\alpha$ and $\Sigma^{\alpha} (\mathcal{A})$ is the set of subsequence sums corresponding to the subsequences of $\mathcal{A}$ that are of the size at most $rk-\alpha$. Then in the same line with the subset sums, we have $|\Sigma_{\alpha} (\mathcal{A})|=|\Sigma^{\alpha} (\mathcal{A})|$ for all $0 \leq \alpha \leq rk$.

The set of subset sums $\Sigma_{\alpha} (A)$ and $\Sigma^{\alpha} (A)$ and the set of subsequence sums $\Sigma_{\alpha} (\mathcal{A})$ and $\Sigma^{\alpha} (\mathcal{A})$ may also be written as unions of sumsets:

For a finite set $A$ of $k$ integers and for positive integers $h, r$, the \emph{$h$-fold sumset} $hA$ is the collection of all sums of $h$ not-necessarily-distinct elements of $A$, the \emph{$h$-fold restricted sumset} $h~\hat{}A$ is the collection of all sums of $h$ distinct elements of $A$, and the \emph{generalized sumset} $h^{(r)} A$ is the collection of all sums of $h$ elements of $A$ with at most $r$ repetitions for each element (see \cite{mistri14}). Then $\Sigma_{\alpha} (A) = \bigcup_{h=\alpha}^{k} h~\hat{}A$,  $\Sigma^{\alpha} (A) = \bigcup_{h=0}^{k-\alpha} h~\hat{}A$, $\Sigma_{\alpha} (\mathcal{A}) = \bigcup_{h=\alpha}^{rk} h^{(r)}A$, and  $\Sigma^{\alpha} (\mathcal{A}) = \bigcup_{h=0}^{rk-\alpha} h^{(r)}A$, where $\mathcal{A} = (A)_{r}$ and $0~\hat{}A = 0^{(r)}A =\{0\}$.

An important problem in additive number theory is to find the optimal lower bound for $|\Sigma_{\alpha} (A)|$ and  $|\Sigma_{\alpha} (\mathcal{A})|$. Such problems are very useful in some other combinatorial problems such as the zero-sum problems (see \cite{devos09, erdos, grynkiewicz1, peng}). Nathanson \cite{nathu95} established the optimal lower bound for $|\Sigma_{1} (A)|$ for sets of integers $A$. Mistri and Pandey \cite{mistri15, mistri16} and Jiang and Li \cite{jiang} extended Nathanson's results to $\Sigma_{1} (\mathcal{A})$ for sequences of integers $\mathcal{A}$. Note that these subset and subsequence sums may also be studied in any abelian group (for earlier works, in case $\alpha=0$ and $\alpha=1$, see \cite{eric12, eric13, devos07, erdos, freeze, griffiths, hamidoune98, hamidoune08, olson}). Recently, Balandraud \cite{eric17} proved the optimal lower bound for $|\Sigma_{\alpha} (A)|$ in the finite prime field $\mathbb{F}_{p}$, where $p$ is a prime number. Inspired by Balandraud's work \cite{eric17}, in this paper we establish optimal lower bounds for $|\Sigma_{\alpha} (A)|$ and $|\Sigma_{\alpha} (\mathcal{A})|$ in the group of integers. Note that, in \cite{bhanja20}, we have already settled this problem when the set $A$ (or sequence $\mathcal{A}$) contains nonnegative or nonpositive integers. So, in this paper we consider the sets (or sequences) which may contain both positive and negative integers.

A similar problem for sumsets also have been extensively studied in the past (see \cite{alon95, alon96, bajnok, cauchy, dav35, dav47, silva94, eliahou, erdos, nathu95, nathu} and the references therein). For some recent works along this line, one can also see \cite{bajnok15, bajnok16, jagannath, jagannath21, bhanja21, bhanja19, mistri14, mistri18, monopoli}.

In Section \ref{S2}, we prove optimal lower bounds for $|\Sigma_{\alpha} (A)|$ for finite sets of integers $A$. In Section \ref{S3}, we extend the results of Section \ref{S2} to sequences of integers.

The following results are used to prove the results in this paper.
\begin{theorem}\textup{\cite[Theorem 1.4]{nathu}\label{sumset-thm}}
	Let $A$, $B$ be nonempty finite sets of integers. Set $A+B = \{a+b: a \in A, b \in B\}$. Then
	\[|A+B| \geq |A|+|B|-1.\]
	This lower bound is optimal.
\end{theorem}

\begin{theorem}\textup{\cite[Theorem 1.3]{nathu}\label{multifold-sumset-thm}}
	Let $A$ be a nonempty finite set of integers and $h$ be a positive integer. Then
	\[|hA| \geq h|A|-h+1.\]
	This lower bound is optimal.
\end{theorem}

\begin{theorem}\textup{\cite[Theorem 4]{eric17} \label{balandraud-thm}}
	Let $A$ be a nonempty subset of $\mathbb{F}_{p}$ such that $A\cap (-A)=\emptyset$. Then for any integer $\alpha \in [0,|A|]$, we have
	\begin{equation*}
		|\Sigma_{\alpha} (A)| \geq \min \left\{p, \frac{|A|(|A|+1)}{2}-\frac{\alpha(\alpha+1)}{2}+1 \right\}.
	\end{equation*}
	This lower bound is optimal.
\end{theorem}

\section{Subset sum}\label{S2}

In Theorem \ref{emptyset-thm} and Corollary \ref{zero-cor}, we prove optimal lower bound for $|\Sigma_{\alpha} (A)|$ under the assumptions $A\cap (-A)=\emptyset$ and $A\cap (-A)=\{0\}$, respectively. In Theorem \ref{improved-subset-thm} and Corollary \ref{improved-subset-cor}, we prove optimal lower bound for $|\Sigma_{\alpha} (A)|$ for arbitrary finite sets of integers $A$. The bounds in Theorem \ref{improved-subset-thm} and Corollary \ref{improved-subset-cor} depends on the number of positive elements and the number of negative elements in set $A$. In Corollary \ref{general-subset-cor}, we prove lower bounds for $|\Sigma_{\alpha} (A)|$, which holds for arbitrary finite sets of integers $A$ and only depend on the total number of elements of $A$ not the number of positive and negative elements of $A$.   

\begin{theorem}\label{emptyset-thm}
	Let $A$ be a set of $k$ integers such that $A\cap (-A)=\emptyset$. For any integer $\alpha \in [0,k]$, we have
	\begin{equation}\label{emptyset-eqn}
		|\Sigma_{\alpha} (A)| \geq \frac{k(k+1)}{2}-\frac{\alpha(\alpha+1)}{2}+1.
	\end{equation}
	This lower bound is optimal.
\end{theorem}

\begin{proof}
	Let $p$ be a prime number such that $p>\max \left\{ 2\max^{\ast}(A),\frac{k(k+1)}{2}-\frac{\alpha(\alpha+1)}{2}+1 \right\}$, where $\max^{\ast}(A)=\max\{|a|: a\in A\}$. Now, the elements of $A$ can be thought of residue classes modulo $p$. Since $p>2\max^{\ast}(A)$, any two elements of $A$ are different modulo $p$. Furthermore $ A \cap (-A)=\emptyset $. Hence, by Theorem \ref{balandraud-thm}, we get
	\[|\Sigma_{\alpha} (A)| \geq \frac{k(k+1)}{2}-\frac{\alpha(\alpha+1)}{2}+1.\]
	
	Next, to verify that the lower bound in (\ref{emptyset-eqn}) is optimal, let $A=[1,k]$. Then $A\cap (-A)=\emptyset$ and
	\begin{align*}
		\Sigma_{\alpha} (A) & \subset \left[1+2+\cdots+\alpha, 1+2+\cdots+k\right]
		= \left[\frac{\alpha(\alpha+1)}{2}, \frac{k(k+1)}{2}\right].
	\end{align*}
	Therefore
	\[|\Sigma_{\alpha} (A)|\leq \frac{k(k+1)}{2}-\frac{\alpha(\alpha+1)}{2}+1.\]
	This together with (\ref{emptyset-eqn}) gives
	\[|\Sigma_{\alpha} (A)|
	=\frac{k(k+1)}{2}-\frac{\alpha(\alpha+1)}{2}+1.\]
	Thus, the lower bound in (\ref{emptyset-eqn}) is optimal. This completes the proof of the theorem.
\end{proof}

\begin{corollary}\label{zero-cor}
	Let $A$ be a set of $k$ integers such that $A\cap (-A)=\{0\}$. For any integer $\alpha\in[0,k]$, we have
	\begin{equation}\label{zero-eqn}
		|\Sigma_{\alpha} (A)| \geq \frac{k(k-1)}{2}-\frac{\alpha(\alpha-1)}{2}+1.
	\end{equation}
	This lower bound is optimal.
\end{corollary}

\begin{proof}
	If $A=\{0\}$, then $\Sigma_{\alpha} (A)=\{0\}$. Therefore $|\Sigma_{\alpha} (A)|=1$, and (\ref{zero-eqn}) holds. So, let $A\neq \{0\}$ and set $A^{\prime}=A\setminus \{0\}$. Then it is easy to see that $\Sigma_{0}(A)=\Sigma_{0}(A^{\prime})$ and $\Sigma_{\alpha}(A)=\Sigma_{\alpha-1}(A^{\prime})$ for $\alpha\geq 1$. Since $A^{\prime}\cap(-A^{\prime})=\emptyset$, by Theorem \ref{emptyset-thm}, we get
	\begin{equation*}
		|\Sigma_{0} (A)|=|\Sigma_{0} (A^{\prime})| \geq \frac{k(k-1)}{2}+1
	\end{equation*}
	and
	\begin{equation*}
		|\Sigma_{\alpha} (A)| = |\Sigma_{\alpha-1} (A^{\prime})|
		\geq \frac{k(k-1)}{2}-\frac{\alpha(\alpha-1)}{2}+1
	\end{equation*}
	for $\alpha\geq 1$. Hence (\ref{zero-eqn}) is established. 
	
	Now, let $A=[0, k-1]$. Then $A\cap (-A)=\{0\}$ and $\Sigma_{\alpha} (A) \subset \left[\frac{\alpha(\alpha-1)}{2}, \frac{k(k-1)}{2}\right]$. Therefore  $|\Sigma_{\alpha} (A)| \leq \frac{k(k-1)}{2}-\frac{\alpha(\alpha-1)}{2}+1$. This together with (\ref{zero-eqn}) gives that the lower bound in (\ref{zero-eqn}) is optimal. 
\end{proof}

\begin{remark}
	Nathanson's theorem \cite[Theorem 3]{nathu95} is a particular case of Theorem \ref{emptyset-thm} and Corollary \ref{zero-cor}, for $\alpha=1$.
\end{remark}

\begin{theorem}\label{improved-subset-thm}
	Let $n$ and $p$ be positive integers and $A$ be a set of $n$ negative and $p$ positive integers. Let $\alpha \in [0, n+p]$ be an integer. 
	\begin{enumerate}[label=\upshape(\roman*)]
		\item If $\alpha \leq n$ and $\alpha \leq p$, then $|\Sigma_{\alpha} (A)| \geq \frac{n(n+1)}{2}+\frac{p(p+1)}{2}+1$.
		
		\item If $\alpha \leq n$ and $\alpha > p$, then $|\Sigma_{\alpha} (A)| \geq \frac{n(n+1)}{2}+\frac{p(p+1)}{2}-\frac{(\alpha-p)(\alpha-p+1)}{2}+1$.
		
		\item If $\alpha > n$ and $\alpha \leq p$, then $|\Sigma_{\alpha} (A)| \geq \frac{n(n+1)}{2}+\frac{p(p+1)}{2}-\frac{(\alpha-n)(\alpha-n+1)}{2}+1$.
		
		\item If $\alpha > n$ and $\alpha > p$, then $|\Sigma_{\alpha} (A)| \geq \frac{n(n+1)}{2}+\frac{p(p+1)}{2}
		-\frac{(\alpha-n)(\alpha-n+1)}{2}-\frac{(\alpha-p)(\alpha-p+1)}{2}+1$.
	\end{enumerate}
	These lower bounds are optimal.
\end{theorem}

\begin{proof}
	Let $A = A_{n} \cup A_{p}$, where $A_{n} = \{b_{1},\ldots,b_{n}\}$ and $A_{p} = \{c_{1},\ldots,c_{p}\}$ such that $b_{n}<b_{n-1}<\cdots<b_{1}<0<c_{1}<c_{2}<\cdots<c_{p}$.
	
	\noindent $\mathtt{(i)}$ If $\alpha \leq n$ and $\alpha \leq p$, then
	\[(\Sigma_{\alpha} (A_{n})+\Sigma_{0} (A_{p})) 
	\cup \left(\Sigma^{1} (\{b_{1}, \ldots, b_{\alpha}\})
	+\sum_{j=1}^{p}c_{j}\right) \subset \Sigma_{\alpha} (A)\]
	with 
	\[(\Sigma_{\alpha} (A_{n})+\Sigma_{0} (A_{p})) \cap \left(\Sigma^{1} (\{b_{1}, \ldots, b_{\alpha}\})+\sum_{j=1}^{p}c_{j}\right) = \emptyset.\]
	Hence, by Theorem \ref{sumset-thm} and Theorem \ref{emptyset-thm}, we have
	\begin{align*}
		|\Sigma_{\alpha} (A)|
		&\geq |\Sigma_{\alpha} (A_{n})+\Sigma_{0} (A_{p})|+|\Sigma^{1} (\{b_{1}, \ldots, b_{\alpha}\})|\\
		&\geq |\Sigma_{\alpha} (A_{n})|+|\Sigma_{0} (A_{p})|+|\Sigma^{1} (\{b_{1}, \ldots, b_{\alpha}\})|-1\\
		&\geq \left(\frac{n(n+1)}{2}-\frac{\alpha(\alpha+1)}{2}+1\right)+\left(\frac{p(p+1)}{2}+1\right)
		+\frac{\alpha(\alpha+1)}{2}-1\\
		&= \frac{n(n+1)}{2}+\frac{p(p+1)}{2}+1.
	\end{align*}
	
	\noindent $\mathtt{(ii)}$ If $\alpha \leq n$ and $\alpha > p$, then
	\[(\Sigma_{\alpha} (A_{n})+\Sigma_{0} (A_{p})) 
	\cup \left(\Sigma_{\alpha-p} (\{b_{1}, \ldots, b_{\alpha}\})
	+\sum_{j=1}^{p}c_{j}\right) \subset \Sigma_{\alpha} (A)\]
	with 
	\[(\Sigma_{\alpha} (A_{n})+\Sigma_{0} (A_{p})) \cap \left(\Sigma_{\alpha-p} (\{b_{1}, \ldots, b_{\alpha}\})
	+\sum_{j=1}^{p}c_{j}\right) = \left\{\sum_{j=1}^{\alpha}b_{j}+\sum_{j=1}^{p} c_{j}\right\}.\] 
	Hence, by Theorem \ref{sumset-thm} and Theorem \ref{emptyset-thm}, we have
	\begin{align*}
		|\Sigma_{\alpha} (A)|
		&\geq |\Sigma_{\alpha} (A_{n})+\Sigma_{0} (A_{p})|+|\Sigma_{\alpha-p} (\{b_{1}, \ldots, b_{\alpha}\})|-1\\
		&\geq |\Sigma_{\alpha} (A_{n})|+|\Sigma_{0} (A_{p})|+|\Sigma_{\alpha-p} (\{b_{1}, \ldots, b_{\alpha}\})|-2\\
		&\geq \left(\frac{n(n+1)}{2}-\frac{\alpha(\alpha+1)}{2}+1\right)
		+\left(\frac{p(p+1)}{2}+1\right)\\
		&\quad+ \left(\frac{\alpha(\alpha+1)}{2}-\frac{(\alpha-p)(\alpha-p+1)}{2}+1\right)-2\\
		&= \frac{n(n+1)}{2}+\frac{p(p+1)}{2}
		-\frac{(\alpha-p)(\alpha-p+1)}{2}+1.
	\end{align*}
	
	\noindent $\mathtt{(iii)}$ If $\alpha > n$ and $\alpha \leq p$, then by applying the result of (ii) for $(-A)$, we obtain 
	\[|\Sigma_{\alpha} (A)| = |\Sigma_{\alpha} (-A)| \geq \frac{n(n+1)}{2}+\frac{p(p+1)}{2}-\frac{(\alpha-n)(\alpha-n+1)}{2}+1.\]
	
	\noindent $\mathtt{(iv)}$ If $\alpha > n$ and $\alpha > p$, then
	\[\left(\sum_{j=1}^{n} b_{j}+\Sigma_{\alpha-n} (A_{p})\right) \cup \left(\Sigma_{\alpha-p} (A_{n})+\sum_{j=1}^{p} c_{j}\right) \subset \Sigma_{\alpha} (A)\]
	with 
	\[\left(\sum_{j=1}^{n} b_{j}+\Sigma_{\alpha-n} (A_{p})\right) \cap \left(\Sigma_{\alpha-p} (A_{n})+\sum_{j=1}^{p} c_{j}\right) = \left\{\sum_{j=1}^{n}b_{j}+\sum_{j=1}^{p} c_{j}\right\}.\] 
	Hence, by Theorem \ref{emptyset-thm}, we get
	\begin{align*}
		|\Sigma_{\alpha} (A)|
		&\geq |\Sigma_{\alpha-n} (A_{p})|+|\Sigma_{\alpha-p} (A_{n})|-1\\
		&\geq \left(\frac{p(p+1)}{2}-\frac{(\alpha-n)(\alpha-n+1)}{2}+1\right) \\
		&\quad+ \left(\frac{n(n+1)}{2}-\frac{(\alpha-p)(\alpha-p+1)}{2}+1\right)-1\\
		&= \frac{n(n+1)}{2}+\frac{p(p+1)}{2}
		-\frac{(\alpha-n)(\alpha-n+1)}{2}-\frac{(\alpha-p)(\alpha-p+1)}{2}+1.
	\end{align*}
	
	It can be easily verified that all the lower bounds mentioned in the theorem are optimal for $A=[-n, p]\setminus\{0\}$.
\end{proof}

\begin{corollary}\label{improved-subset-cor}
	Let $n$ and $p$ be positive integers and $A$ be a set of $n$ negative integers, $p$ positive integers and zero. Let $\alpha \in [0, n+p+1]$ be an integer. 
	\begin{enumerate}[label=\upshape(\roman*)]
		\item If $\alpha \leq n$ and $\alpha \leq p$, then $|\Sigma_{\alpha} (A)| \geq \frac{n(n+1)}{2}+\frac{p(p+1)}{2}+1$.
		
		\item If $\alpha \leq n$ and $\alpha > p$, then $|\Sigma_{\alpha} (A)| \geq \frac{n(n+1)}{2}+\frac{p(p+1)}{2}-\frac{(\alpha-p)(\alpha-p-1)}{2}+1$.
		
		\item If $\alpha > n$ and $\alpha \leq p$, then $|\Sigma_{\alpha} (A)| \geq \frac{n(n+1)}{2}+\frac{p(p+1)}{2}-\frac{(\alpha-n)(\alpha-n-1)}{2}+1$.
		
		\item If $\alpha > n$ and $\alpha > p$, then $|\Sigma_{\alpha} (A)| \geq \frac{n(n+1)}{2}+\frac{p(p+1)}{2}
		-\frac{(\alpha-n)(\alpha-n-1)}{2}-\frac{(\alpha-p)(\alpha-p-1)}{2}+1$.
	\end{enumerate}
	These lower bounds are optimal.
\end{corollary}

\begin{proof}
	The lower bounds for $|\Sigma_{\alpha} (A)|$ easily follows from Theorem \ref{improved-subset-thm} and the fact that $\Sigma_{0} (A) = \Sigma_{0} (A^\prime)$ and $\Sigma_{\alpha} (A) = \Sigma_{\alpha-1} (A^\prime)$ for $\alpha \geq 1$, where $A^\prime = A\setminus\{0\}$. Furthermore, the optimality of these bounds can be verified by taking $A=[-n, p]$.
\end{proof}

\begin{corollary}\label{general-subset-cor}
	Let $k \geq 2$ and $A$ be a set of $k$ integers. Let $\alpha \in [0, k]$ be an integer. If $0\notin A$, then
	\begin{equation}\label{general-subset-eqn1}
		|\Sigma_{\alpha} (A)| \geq \left\lfloor \frac{(k+1)^{2}}{4} \right\rfloor-\frac{\alpha(\alpha+1)}{2}+1.
	\end{equation}
	If $0\in A$, then
	\begin{equation}\label{general-subset-eqn2}
		|\Sigma_{\alpha} (A)| \geq \left\lfloor \frac{k^{2}}{4} \right\rfloor-\frac{\alpha(\alpha-1)}{2}+1.
	\end{equation}
\end{corollary}

\begin{proof}
	\noindent $\mathtt{Case~1}$. ($0 \notin A$). If $k=2$, then $A=\{a_{1}, a_{2}\}$ for some integers $a_{1}, a_{2}$ with $a_{1}<a_{2}$. Therefore $\Sigma_{1} (A) = \{a_{1}, a_{2}, a_{1}+a_{2}\} \subset \Sigma_{0} (A)$ and $\Sigma_{2} (A)=\{a_{1}+a_{2}\}$. Hence (\ref{general-subset-eqn1}) holds for $k=2$. So, assume that $k\geq 3$. As $k(k+1)/2 > (k+1)^{2}/4$ for $k\geq 3$, if $|\Sigma_{\alpha} (A)| \geq k(k+1)/2-\alpha(\alpha+1)/2+1$, then we are done. So, let $|\Sigma_{\alpha} (A)| < k(k+1)/2-\alpha(\alpha+1)/2+1$. Then, Theorem \ref{emptyset-thm} implies that $A$ contains both positive and negative integers. Let $A_n$ and $A_p$ be subsets of $A$ that contain respectively, all negative and all positive integers of $A$. Let also $|A_n| = n$ and $|A_p| = p$. Then $n\geq 1$ and $p \geq 1$. By Theorem \ref{improved-subset-thm}, we have 
	\[|\Sigma_{\alpha} (A)| \geq  \frac{n(n+1)}{2}+\frac{p(p+1)}{2}-\frac{\alpha(\alpha+1)}{2}+1\]
	for all $\alpha \in [0, k]$. Since $k=n+p$, without loss of generality we may assume that $n\geq k/2$. Therefore
	\begin{align*}
		|\Sigma_{\alpha} (A)|
		&\geq \frac{n(n+1)}{2}+\frac{(k-n)(k-n+1)}{2}-\frac{\alpha(\alpha+1)}{2}+1\\
		&= \left(n-\frac{k}{2}\right)^{2}+\frac{(k+1)^2-1}{4}-\frac{\alpha(\alpha+1)}{2} + 1\\
		&\geq \frac{(k+1)^2-1}{4}-\frac{\alpha(\alpha+1)}{2}+1.
	\end{align*}
	Hence
	\[|\Sigma_{\alpha} (A)| \geq \left\lfloor \frac{(k+1)^{2}}{4} \right\rfloor-\frac{\alpha(\alpha+1)}{2}+1.\]
	
	\noindent $\mathtt{Case~2}$. ($0 \in A$). Set $A^{\prime}=A\setminus \{0\}$. Then $\Sigma_{0}(A) = \Sigma_{0}(A^{\prime})$ and $\Sigma_{\alpha}(A) = \Sigma_{\alpha-1}(A^{\prime})$ for $\alpha\geq 1$. Hence, by Case 1, we get
	\begin{equation*}
		|\Sigma_{0}(A)| 
		= |\Sigma_{0}(A^{\prime})| 
		\geq \left\lfloor \frac{k^{2}}{4} \right\rfloor+1
	\end{equation*}
	and
	\begin{align*}
		|\Sigma_{\alpha} (A)|
		= |\Sigma_{\alpha-1} (A^{\prime})|
		\geq \left\lfloor \frac{k^{2}}{4} \right\rfloor
		-\frac{\alpha(\alpha-1)}{2}+1
	\end{align*}
	for $\alpha \geq 1$. Hence
	\[|\Sigma_{\alpha} (A)| \geq \left\lfloor \frac{k^{2}}{4} \right\rfloor-\frac{\alpha(\alpha-1)}{2}+1\]
	for all $\alpha \in [0, k]$. This completes the proof of the corollary.
\end{proof}

\begin{remark}
	Nathanson \cite{nathu95} have already proved this corollary for $\alpha=1$. The purpose of this corollary is to prove a similar result for every $\alpha \in [0, k]$. Note that the lower bounds in Corollary \ref{general-subset-cor} are not optimal for all $\alpha \in [0, k]$, except for $\alpha=0$ and $\alpha=1$.
\end{remark}

\begin{remark}
	The lower bounds in Corollary \ref{general-subset-cor} can also be written in the following form:\\	
	If $0 \notin A$, then
	\begin{equation*}
		|\Sigma_{\alpha} (A)| \geq 
		\begin{cases}
			\frac{(k+1)^2}{4}-\frac{\alpha(\alpha+1)}{2}+1 & \text{if } k \equiv 1 \pmod{2}\\
			\frac{(k+1)^2-1}{4}-\frac{\alpha(\alpha+1)}{2}+1 & \text{if } k \equiv 0 \pmod{2}.
		\end{cases}
	\end{equation*}
	If $0\in A$, then
	\begin{equation*}
		|\Sigma_{\alpha} (A)| \geq
		\begin{cases}
			\frac{k^2-1}{4}-\frac{\alpha(\alpha-1)}{2}+1 & \text{ if } k \equiv 1 \pmod{2}\\
			\frac{k^2}{4}-\frac{\alpha(\alpha-1)}{2}+1 & \text{ if } k \equiv 0 \pmod{2}.
		\end{cases}
	\end{equation*}
\end{remark}

\section{Subsequence sum}\label{S3}

In this section, we extend the results of the previous section from sets of integers to sequences of integers. In Theorem \ref{restriction-subsequence-thm}, we establish optimal lower bound for $|\Sigma_{\alpha} (\mathcal{A})|$ under the assumptions $\mathcal{A}\cap (-\mathcal{A})=\emptyset$ and $\mathcal{A}\cap (-\mathcal{A})=(0)_{r}$. In Theorem \ref{improved-subsequence-thm} and Corollary \ref{improved-subsequence-cor}, we prove optimal lower bound for $|\Sigma_{\alpha} (\mathcal{A})|$ for arbitrary finite sequences of integers $\mathcal{A}$. The bounds in Theorem \ref{improved-subsequence-thm} and Corollary \ref{improved-subsequence-cor} depends on the number of negative terms and the number of positive terms in sequence $\mathcal{A}$. In Corollary \ref{general-subsequence-cor}, we prove lower bounds for $|\Sigma_{\alpha} (\mathcal{A})|$, which holds for arbitrary finite sequences of integers $\mathcal{A}$ and only depend on the total number of terms of $\mathcal{A}$ not the number of positive and negative terms of $\mathcal{A}$.

If $\alpha=rk$ and $\mathcal{A}=(a_{1}, \ldots, a_{k})_{r}$, then $\Sigma_{\alpha} (\mathcal{A})=\{ra_{1}+ra_{2}+\cdots+ra_{k}\}$. Therefore  $|\Sigma_{\alpha} (\mathcal{A})|=1$. So, in the rest of this section, we assume that $\alpha < rk$.

\begin{theorem}\label{restriction-subsequence-thm}
	Let $k\geq 2$, $r\geq 1$, and $\alpha\in[0,rk-1]$ be integers. Let $m\in [1,k]$ be an integer such that $(m-1)r\leq \alpha<mr$. Let $\mathcal{A}$ be a sequence of $rk$ terms which is made up of $k$ distinct integers each repeated exactly $r$ times. If $\mathcal{A}\cap(-\mathcal{A})=\emptyset$, then
	\begin{equation}\label{moregeneral-subsequence-eqn1}
		|\Sigma_{\alpha} (\mathcal{A})| \geq r\left(\frac{k(k+1)}{2}-\frac{m(m+1)}{2}\right)+m(mr-\alpha)+1.
	\end{equation}
	If $\mathcal{A}\cap(-\mathcal{A})=(0)_{r}$, then
	\begin{equation}\label{moregeneral-subsequence-eqn2}
		|\Sigma_{\alpha} (\mathcal{A})| 
		\geq r\left(\frac{k(k-1)}{2}-\frac{m(m-1)}{2}\right)
		+(m-1)(mr-\alpha)+1.
	\end{equation}
	These lower bounds are optimal. 
\end{theorem}

\begin{proof}
	Let $A$ be the set of all distinct terms of sequence $\mathcal{A}$. Since $(m-1)r\leq \alpha<mr$, we can write $\alpha$ as $\alpha=(m-1)r+u$ for some integer $0 \leq u < r$. Then
	\[(r-u)\Sigma_{m-1} (A) + u\Sigma_{m} (A) 
	\subset \Sigma_{\alpha} (\mathcal{A}),\]
	where $(r-u)\Sigma_{m-1} (A)$ is the $(r-u)$-fold sumset of $\Sigma_{m-1} (A)$ and $u\Sigma_{m} (A)$ is the $u$-fold sumset of $\Sigma_{m} (A)$. So, by Theorem \ref{sumset-thm} and Theorem \ref{multifold-sumset-thm}, we have
	\begin{align*}
		|\Sigma_{\alpha} (\mathcal{A})| 
		\geq |(r-u)\Sigma_{m-1} (A)|+|u\Sigma_{m} (A)|-1
		\geq (r-u)|\Sigma_{m-1} (A)|+u|\Sigma_{m} (A)|-r+1.
	\end{align*}
	If $\mathcal{A}\cap(-\mathcal{A})=\emptyset$, then $A\cap(-A)=\emptyset$. Thus, by Theorem \ref{emptyset-thm}, we have
	\begin{align*}
		|\Sigma_{\alpha} (\mathcal{A})|
		&\geq (r-u)\left(\frac{k(k+1)}{2}-\frac{m(m-1)}{2}+1\right)
		+u\left(\frac{k(k+1)}{2}-\frac{m(m+1)}{2}+1\right)
		-r+1\\
		&= r\left(\frac{k(k+1)}{2}-\frac{m(m+1)}{2}\right)
		+m(r-u)+1\\
		&= r\left(\frac{k(k+1)}{2}-\frac{m(m+1)}{2}\right)
		+m(mr-\alpha)+1.
	\end{align*}
	Similarly, if $\mathcal{A}\cap(-\mathcal{A})=(0)_{r}$, then $A\cap(-A)=\{0\}$. Thus, by Corollary \ref{zero-cor}, we have
	\begin{align*}
		|\Sigma_{\alpha} (\mathcal{A})|
		&\geq (r-u)\left(\frac{k(k-1)}{2}-\frac{(m-1)(m-2)}{2}+1\right)+u\left(\frac{k(k-1)}{2}-\frac{m(m-1)}{2}+1\right)\\
		&\quad -r+1\\
		&= r\left(\frac{k(k-1)}{2}-\frac{m(m-1)}{2}\right)
		+(m-1)(r-u)+1\\
		&= r\left(\frac{k(k-1)}{2}-\frac{m(m-1)}{2}\right)
		+(m-1)(mr-\alpha)+1.
	\end{align*}
	Hence (\ref{moregeneral-subsequence-eqn1}) and (\ref{moregeneral-subsequence-eqn2}) are established.
	
	Next, to verify that the lower bounds in (\ref{moregeneral-subsequence-eqn1}) and (\ref{moregeneral-subsequence-eqn2}) are optimal, let $\mathcal{A}=[1, k]_{r}$ and $\mathcal{B}=[0, k-1]_{r}$. Then $\mathcal{A}\cap (-\mathcal{A})=\emptyset$ and $\mathcal{B}\cap (-\mathcal{B})=(0)_{r}$ with
	\begin{align*}
		\Sigma_{\alpha} (\mathcal{A}) \subset \left[r\cdot1+\cdots+r\cdot(m-1)+(\alpha-(m-1)r)\cdot m, r\cdot1+\cdots+r\cdot k\right]
	\end{align*}
	and 
	\begin{align*}
		\Sigma_{\alpha} (\mathcal{B}) \subset \left[r\cdot1+\cdots+r\cdot(m-2)+(\alpha-(m-1)r)\cdot (m-1), r\cdot1+\cdots+r\cdot (k-1)\right].
	\end{align*}
	Therefore
	\[|\Sigma_{\alpha} (\mathcal{A})| \leq \frac{rk(k+1)}{2}-\frac{rm(m+1)}{2}+m(mr-\alpha)+1\]
	and
	\[|\Sigma_{\alpha} (\mathcal{B})| \leq \frac{rk(k-1)}{2}-\frac{rm(m-1)}{2}+(m-1)(mr-\alpha)+1.\]
	These two inequalities together with (\ref{moregeneral-subsequence-eqn1}) and (\ref{moregeneral-subsequence-eqn2}) implies that the lower bounds in (\ref{moregeneral-subsequence-eqn1}) and (\ref{moregeneral-subsequence-eqn2}) are optimal. This completes the proof of the theorem.
\end{proof}

\begin{remark}
	Mistri and Pandey's result \cite[Theorem 1]{mistri16} is a particular case of Theorem \ref{restriction-subsequence-thm}, for $\alpha=1$.
\end{remark}

\begin{theorem}\label{improved-subsequence-thm}
	Let $k\geq 2$, $r\geq 1$, and $\alpha\in[0,rk-1]$ be integers. Let $m\in [1,k]$ be an integer such that $(m-1)r\leq \alpha<mr$. Let $\mathcal{A}$ be a sequence of $rk$ terms which is made up of $n$ negative integers and $p$ positive integers each repeated exactly $r$ times.  	
	\begin{enumerate}[label=\upshape(\roman*)]
		\item If $m \leq n$ and $m \leq p$, then 
		$|\Sigma_{\alpha} (\mathcal{A})| \geq r\left(\frac{n(n+1)}{2}+\frac{p(p+1)}{2}\right)+1$.
		
		\item If $m \leq n$ and $m > p$, then
		$|\Sigma_{\alpha} (\mathcal{A})| 
		\geq r\left(\frac{n(n+1)}{2}+\frac{p(p+1)}{2}
		-\frac{(m-p)(m-p+1)}{2}\right)+(m-p)(mr-\alpha)+1$.
		
		\item If $m > n$ and $m \leq p$, then
		$|\Sigma_{\alpha} (\mathcal{A})| 
		\geq r\left(\frac{n(n+1)}{2}+\frac{p(p+1)}{2}
		-\frac{(m-n)(m-n+1)}{2}\right)+(m-n)(mr-\alpha)+1$.
		
		\item If $m > n$ and $m > p$, then 
		$|\Sigma_{\alpha} (\mathcal{A})| \geq r\left(\frac{n(n+1)}{2}+\frac{p(p+1)}{2}-\frac{(m-n)(m-n+1)}{2}
		-\frac{(m-p)(m-p+1)}{2}\right)+(2m-n-p)(mr-\alpha)+1$.
	\end{enumerate}
	These lower bounds are optimal.
\end{theorem}

\begin{proof}
	Let $A_n$ and $A_p$ be sets that contain respectively, all distinct negative terms and all distinct positive terms of  $\mathcal{A}$. Then $|A_n|=n$ and $|A_p|=p$. Let also $A_{n}=\{b_{1}, b_{2}, \ldots, b_{n}\}$ and $A_{p}=\{c_{1}, c_{2}, \ldots,c_{p}\}$, where $b_{n}<b_{n-1}<\cdots<b_{1}<0<c_{1}<c_{2}<\cdots<c_{p}$.

	\noindent $\mathtt{(i)}$ If $m \leq n$ and $m \leq p$, then
	\[r(\Sigma_{m} (A_{n})+\Sigma_{0} (A_{p})) 
	\cup \left(\Sigma^{1}((b_{1}, \ldots, b_{m})_{r})+\sum_{j=1}^{p} rc_{j}\right) \subset \Sigma_{\alpha} (\mathcal{A})\]
	with 
	\[r(\Sigma_{m} (A_{n})+\Sigma_{0} (A_{p})) 
	\cap \left(\Sigma^{1}((b_{1}, \ldots, b_{m})_{r})+\sum_{j=1}^{p} rc_{j}\right) = \emptyset.\] 
	Hence, by Theorem \ref{sumset-thm}, Theorem \ref{multifold-sumset-thm}, Theorem \ref{emptyset-thm}, and Theorem \ref{restriction-subsequence-thm}, we have
	\begin{align*}
		|\Sigma_{\alpha} (\mathcal{A})|
		&\geq |r(\Sigma_{m} (A_{n})+\Sigma_{0} (A_{p}))|
		+|\Sigma^{1}((b_{1}, \ldots, b_{m})_{r})|\\
		&\geq r|\Sigma_{m} (A_{n})|+r|\Sigma_{0} (A_{p})|+|\Sigma^{1}((b_{1}, \ldots, b_{m})_{r})|-2r+1\\
		&\geq r\left(\frac{n(n+1)}{2}-\frac{m(m+1)}{2}+1\right)
		+r\left(\frac{p(p+1)}{2}+1\right)+\frac{rm(m+1)}{2}-2r+1\\
		&= r\left(\frac{n(n+1)}{2}+\frac{p(p+1)}{2}\right)+1.
	\end{align*}

	\noindent $\mathtt{(ii)}$ If $m \leq n$ and $m > p$, then
	\[r(\Sigma_{m} (A_{n})+\Sigma_{0} (A_{p})) 
	\cup \left(\Sigma_{\alpha-pr} ((b_{1}, \ldots, b_{m})_{r})
	+\sum_{j=1}^{p}rc_{j}\right)
	\subset \Sigma_{\alpha} (\mathcal{A})\]
	with 
	\[r(\Sigma_{m} (A_{n})+\Sigma_{0} (A_{p})) \cap \left(\Sigma_{\alpha-pr} ((b_{1}, \ldots, b_{m})_{r})+\sum_{j=1}^{p}rc_{j}\right) 
	= \left\{\sum_{j=1}^{m}rb_{j}+\sum_{j=1}^{p}rc_{j}\right\}.\]  
	Hence, by Theorem \ref{sumset-thm}, Theorem \ref{multifold-sumset-thm}, Theorem \ref{emptyset-thm}, and Theorem \ref{restriction-subsequence-thm}, we have
	\begin{align*}
		|\Sigma_{\alpha} (\mathcal{A})|
		&\geq |r(\Sigma_{m} (A_{n})+\Sigma_{0} (A_{p}))|
		+|\Sigma_{\alpha-pr} ((b_{1}, \ldots, b_{m})_{r})|-1\\
		&\geq r|\Sigma_{m} (A_{n})|+r|\Sigma_{0} (A_{p})|+|\Sigma_{\alpha-pr} ((b_{1}, \ldots, b_{m})_{r})|-2r+1\\
		&\geq r\left(\frac{n(n+1)}{2}-\frac{m(m+1)}{2}+1\right)+r\left(\frac{p(p+1)}{2}+1\right)\\
		&\quad+ r\left(\frac{m(m+1)}{2}-\frac{(m-p)(m-p+1)}{2}\right)+(m-p)(mr-\alpha)-2r+1\\
		&= r\left(\frac{n(n+1)}{2}+\frac{p(p+1)}{2}
		-\frac{(m-p)(m-p+1)}{2}\right)+(m-p)(mr-\alpha)+1.
	\end{align*}

	\noindent $\mathtt{(iii)}$ If $m > n$ and $m \leq p$, then by applying the result of (ii) for $(-\mathcal{A})$, we obtain 
	\[|\Sigma_{\alpha} (\mathcal{A})| 
	= |\Sigma_{\alpha} (-\mathcal{A})| 
	\geq r\left(\frac{n(n+1)}{2}+\frac{p(p+1)}{2}
	-\frac{(m-n)(m-n+1)}{2}\right)+(m-n)(mr-\alpha)+1.\]

	\noindent $\mathtt{(iv)}$ If $m > n$ and $m > p$, then
	\begin{align*}
		\left(\sum_{j=1}^{n}rb_{j}+\Sigma_{\alpha-nr} ((A_{p})_{r})\right) 
		\cup \left(\Sigma_{\alpha-pr} ((A_{n})_{r})
		+\sum_{j=1}^{p}rc_{j}\right)
		\subset \Sigma_{\alpha} (\mathcal{A})
	\end{align*}
	with
	\[\left(\sum_{j=1}^{n}rb_{j}+\Sigma_{\alpha-nr} ((A_{p})_{r})\right) 
	\cap \left(\Sigma_{\alpha-pr} ((A_{n})_{r})
	+\sum_{j=1}^{p}rc_{j}\right) 
	= \left\{\sum_{j=1}^{n}rb_{j}+\sum_{j=1}^{p}rc_{j}\right\}.\] 
	Hence, by Theorem \ref{restriction-subsequence-thm}, we have
	\begin{align*}
		|\Sigma_{\alpha} (\mathcal{A})|
		&\geq |\Sigma_{\alpha-nr} ((A_{p})_{r})|
		+|\Sigma_{\alpha-pr} ((A_{n})_{r})|-1\\
		&\geq r\left(\frac{p(p+1)}{2}-\frac{(m-n)(m-n+1)}{2}\right)
		+(m-n)(mr-\alpha)\\
		&\quad +r\left(\frac{n(n+1)}{2}-\frac{(m-p)(m-p+1)}{2}\right)
		+(m-p)(mr-\alpha)+1\\
		&= r\left(\frac{n(n+1)}{2}+\frac{p(p+1)}{2}-\frac{(m-n)(m-n+1)}{2}
		-\frac{(m-p)(m-p+1)}{2}\right)\\
		&\quad +(2m-n-p)(mr-\alpha)+1.
	\end{align*}
	
	Furthermore, the optimality of the lower bounds in (i)--(iv) can be verified by taking $\mathcal{A}=[-n, p]_{r}\setminus(0)_{r}$.
\end{proof}

\begin{corollary}\label{improved-subsequence-cor}
	Let $k\geq 2$, $r\geq 1$, and $\alpha\in[0,rk-1]$ be integers. Let $m\in [1,k]$ be an integer such that $(m-1)r\leq \alpha<mr$. Let $\mathcal{A}$ be a sequence of $rk$ terms which is made up of $n$ negative integers, $p$ positive integers and zero, each repeated exactly $r$ times.  
	\begin{enumerate}[label=\upshape(\roman*)]
		\item If $m \leq n$ and $m \leq p$, then $|\Sigma_{\alpha} (\mathcal{A})| \geq r\left(\frac{n(n+1)}{2}+\frac{p(p+1)}{2}\right)+1$.
		
		\item If $m \leq n$ and $m > p$, then
		$|\Sigma_{\alpha} (\mathcal{A})| 
		\geq r\left(\frac{n(n+1)}{2}+\frac{p(p+1)}{2}
		-\frac{(m-p)(m-p-1)}{2}\right)+(m-p-1)(mr-\alpha)+1$.
		
		\item If $m > n$ and $m \leq p$, then
		$|\Sigma_{\alpha} (\mathcal{A})| 
		\geq r\left(\frac{n(n+1)}{2}+\frac{p(p+1)}{2}
		-\frac{(m-n)(m-n-1)}{2}\right)+(m-n-1)(mr-\alpha)+1$.
		
		\item If $m > n$ and $m > p$, then 
		$|\Sigma_{\alpha} (\mathcal{A})| 
		\geq r\left(\frac{n(n+1)}{2}+\frac{p(p+1)}{2}
		-\frac{(m-n)(m-n-1)}{2}-\frac{(m-p)(m-p-1)}{2}\right)+(2m-n-p-2)(mr-\alpha)+1$.
	\end{enumerate}
	These lower bounds are optimal.
\end{corollary}

\begin{proof}
	The lower bounds for $|\Sigma_{\alpha} (\mathcal{A})|$ easily follows from Theorem \ref{improved-subsequence-thm} and the fact that $\Sigma_{\alpha} (\mathcal{A}) = \Sigma_{0} (\mathcal{A}^\prime)$ for $0 \leq \alpha < r$ and $\Sigma_{\alpha} (\mathcal{A}) = \Sigma_{\alpha-r} (\mathcal{A}^\prime)$ for $r \leq \alpha < rk$, where $\mathcal{A^{\prime}}=\mathcal{A}\setminus(0)_{r}$. Furthermore, the optimality of these bounds can be verified by taking $\mathcal{A}=[-n, p]_{r}$.
\end{proof}

\begin{corollary}\label{general-subsequence-cor}
	Let $k\geq 3$, $r\geq 1$, and $\alpha\in[0,rk-1]$ be integers. Let $m\in [1,k]$ be an integer such that $(m-1)r\leq \alpha<mr$. Let $\mathcal{A}$ be a sequence of $rk$ terms which is made up of $k$ distinct integers each repeated exactly $r$ times. If $0\notin \mathcal{A}$, then
	\begin{equation*}
		|\Sigma_{\alpha} (\mathcal{A})| \geq
		\begin{cases}
			r\left(\frac{(k+1)^2}{4}-\frac{m(m+1)}{2}\right)+1 & \text{if } k \equiv 1 \pmod{2}\\
			r\left(\frac{(k+1)^2-1}{4}-\frac{m(m+1)}{2}\right)+1 & \text{if } k \equiv 0 \pmod{2}.
		\end{cases}
	\end{equation*}
	If $0\in \mathcal{A}$, then
	\begin{equation*}
		|\Sigma_{\alpha} (\mathcal{A})| \geq
		\begin{cases}
			r\left(\frac{k^2-1}{4}-\frac{m(m-1)}{2}\right)+1 & \text{if } k \equiv 1 \pmod{2}\\
			r\left(\frac{k^2}{4}-\frac{m(m-1)}{2}\right)+1 & \text{if } k \equiv 0 \pmod{2}.
		\end{cases}
	\end{equation*}
\end{corollary}

\begin{proof}
	Note that
	\begin{equation*}
		\frac{rk(k+1)}{2} \geq
		\begin{cases}
			\frac{r(k+1)^2}{4}+1 & \text{if } k \equiv 1 \pmod{2}\\
			\frac{r((k+1)^2-1)}{4}+1 & \text{if } k \equiv 0 \pmod{2}
		\end{cases}
	\end{equation*}
	and
	\begin{equation*}
		\frac{rk(k-1)}{2}+1 \geq
		\begin{cases}
			\frac{r(k^2-1)}{4}+1 & \text{if } k \equiv 1 \pmod{2}\\
			\frac{rk^2}{4}+1 & \text{if } k \equiv 0 \pmod{2}
		\end{cases}
	\end{equation*}
	for $k\geq 3$. If $0 \notin \mathcal{A}$ and $|\Sigma_{\alpha} (\mathcal{A})| \geq r\left(\frac{k(k+1)}{2}-\frac{m(m+1)}{2}\right)+m(mr-\alpha)+1$, then we are done. So, let $|\Sigma_{\alpha} (\mathcal{A})| < r\left(\frac{k(k+1)}{2}-\frac{m(m+1)}{2}\right)+m(mr-\alpha)+1$ when $0 \notin \mathcal{A}$. Then, Theorem \ref{restriction-subsequence-thm} implies that $\mathcal{A}$ contains both positive and negative integers. By similar arguments, when $0 \in \mathcal{A}$ also, we can assume that $\mathcal{A}$ contains both positive and negative integers. So, in both the cases $0 \in \mathcal{A}$ and $0 \notin \mathcal{A}$, we can assume that $\mathcal{A}$ contains both positive and negative integers.	Let $A_n$ and $A_p$ be sets that contain respectively, all distinct negative terms and all distinct positive terms of sequence $\mathcal{A}$. Let also $|A_n| = n$ and $|A_p| = p$. Then $n \geq 1$ and $p \geq 1$.
	
	\noindent $\mathtt{Case~1}$. ($0\notin \mathcal{A}$). By Theorem \ref{improved-subsequence-thm}, we have
	\[|\Sigma_{\alpha} (\mathcal{A})|
	\geq r\left(\frac{n(n+1)}{2}+\frac{p(p+1)}{2} -\frac{m(m+1)}{2}\right)+1\]
	for all $\alpha \in [0, rk-1]$. Therefore
	\begin{align*}
		|\Sigma_{\alpha} (\mathcal{A})|
		&\geq r\left(\frac{n(n+1)}{2}+\frac{(k-n)(k-n+1)}{2}
		-\frac{m(m+1)}{2}\right)+1\\
		&= r\left(\left(n-\frac{k}{2}\right)^{2}+\frac{k^2+2k}{4}-\frac{m(m+1)}{2}\right)+1.
	\end{align*}
	Since $k=n+p$, without loss of generality we may assume that $n\geq \lceil k/2\rceil$. If $k \equiv 1 \pmod{2}$, then $k=2t+1$ for some positive integer $t$. Hence
	\begin{align*}
		|\Sigma_{\alpha} (\mathcal{A})|
		&\geq r\left(\left(n-t-\frac{1}{2}\right)^{2}+\frac{k^2+2k}{4}-\frac{m(m+1)}{2}\right)+1\\
		&= r\left((n-t)(n-t-1)+\frac{(k+1)^2}{4}-\frac{m(m+1)}{2}\right)+1\\
		&\geq r\left(\frac{(k+1)^2}{4}-\frac{m(m+1)}{2}\right)+1.
	\end{align*}
	If $k \equiv 0 \pmod{2}$, then $k=2t$ for some positive integer $t$. Without loss of generality we may assume that $n \geq t$. Hence
	\begin{align*}
		|\Sigma_{\alpha} (\mathcal{A})|
		&\geq r\left((n-t)^{2}+\frac{k^2+2k}{4}-\frac{m(m+1)}{2}\right)+1\\
		&\geq r\left(\frac{(k+1)^2-1}{4}-\frac{m(m+1)}{2}\right)+1.
	\end{align*}

	\noindent $\mathtt{Case~2}$. ($0 \in \mathcal{A}$). By Corollary \ref{improved-subsequence-cor}, we have
	\[|\Sigma_{\alpha} (\mathcal{A})|
	\geq r\left(\frac{n(n+1)}{2}+\frac{p(p+1)}{2} -\frac{m(m-1)}{2}\right)+1\]
	for all $\alpha \in [0, rk-1]$. Therefore
	\begin{align*}
		|\Sigma_{\alpha} (\mathcal{A})|
		&\geq r\left(\left(n-\frac{k-1}{2}\right)^{2}+\frac{k^2-1}{4}-\frac{m(m-1)}{2}\right)+1
	\end{align*}
	Since $k=n+p+1$, without loss of generality we may assume that $n\geq \lceil (k-1)/2 \rceil$. If $k \equiv 1 \pmod{2}$, then $k=2t+1$ for some positive integer $t$. Hence
	\begin{align*}
		|\Sigma_{\alpha} (\mathcal{A})|
		&\geq r\left((n-t)^2+\frac{k^2-1}{4}-\frac{m(m-1)}{2}\right)+1\\
		&\geq r\left(\frac{k^2-1}{4}-\frac{m(m-1)}{2}\right)+1.
	\end{align*}
	If $k \equiv 0 \pmod{2}$, then $k=2t$ for some positive integer $t$. Hence
	\begin{align*}
		|\Sigma_{\alpha} (\mathcal{A})|
		&\geq r\left(\left(n-t+\frac{1}{2}\right)^{2}+\frac{k^2-1}{4}-\frac{m(m-1)}{2}\right)+1\\
		&= r\left((n-t)(n-t+1)+\frac{k^2}{4}-\frac{m(m-1)}{2}\right)+1\\
		&\geq  r\left(\frac{k^2}{4}-\frac{m(m-1)}{2}\right)+1.
	\end{align*}
	This completes the proof of the corollary.
\end{proof}

\begin{remark}
	Mistri and Pandey \cite{mistri16} have already proved this corollary for $\alpha=1$. The purpose of this corollary is to prove a similar result for every $\alpha \in [0, rk-1]$. Note that the lower bounds in Corollary \ref{general-subsequence-cor} are not optimal for all $\alpha \in [0, rk-1]$, except for $\alpha=0$ and $\alpha=1$.
\end{remark}

\section{Open problems}
\begin{enumerate}
	\item Along this line, it is important to find the optimal lower bound for $|\Sigma_{\alpha} (\mathcal{A})|$, for arbitrary finite sequence of integers $\mathcal{A} = (\underbrace{a_{1},\ldots,a_{1}}_{r_{1}~\text{copies}}, \underbrace{a_{2},\ldots,a_{2}}_{r_{2}~\text{copies}},\ldots, \underbrace{a_{k},\ldots,a_{k}}_{r_{k}~\text{copies}})$. When the sequence $\mathcal{A}$ contains nonnegative or nonpositive integers, we already have the optimal lower bound for $|\Sigma_{\alpha} (\mathcal{A})|$ (see \cite{bhanja20}). So, the only case that remains to study is when the sequence $\mathcal{A}$ contains both positive and negative integers. Note that, in this paper we settled this problem in the special case $r_{i}=r$ for all $i=1,2, \ldots, k$. 
	
	\item It is also an important problem to study the structure of the sequence $\mathcal{A}$ for which the lower bound for $|\Sigma_{\alpha} (\mathcal{A})|$ is optimal. When $\mathcal{A}$ contains nonnegative or nonpositive integers this problem has already been established (see \cite{bhanja20}). So, it remains to solve this problem when the sequence $\mathcal{A}$ contains both positive and negative integers. 
	
	\item For a finite set $H$ of nonnegative integers and a finite set $A$ of $k$ integers, define the sumsets	
	\[HA := \bigcup_{h \in H} hA,~~ H~\hat{}A := \bigcup_{h \in H} h~\hat{}A \text{ and } H^{(r)}A := \bigcup_{h \in H} h^{(r)}A.\] 
	Then $H~\hat{}A = \Sigma_{\alpha}(A)$ for $H = [\alpha, k]$, $H~\hat{}A = \Sigma^{\alpha}(A)$ for $H = [0, k-\alpha]$, $H^{(r)}A = \Sigma_{\alpha}(\mathcal{A})$ for $H = [\alpha, rk]$, and $H^{(r)}A = \Sigma^{\alpha}(\mathcal{A})$ for $H = [0, rk-\alpha]$, where $\mathcal{A} = (A)_{r}$. Along the same line with the sumsets $hA$, $h~\hat{}A$, and $\Sigma_{\alpha}(A)$, the first author established optimal lower bounds for $|HA|$ and $|H~\hat{}A|$, when $A$ contains nonnegative or nonpositive integers (see \cite{jagannath}). The author also characterized the sets $H$ and $A$ for which the lower bounds are achieved \cite{jagannath}. It will be interesting to generalize such results to the sumset $H^{(r)}A$. 
\end{enumerate}

	\section*{Acknowledgment}
	During the first author's time at the Indian Institute of Technology Roorkee and then at the Harish-Chandra Research Institute in Prayagraj, this work was completed. For both research facilities and financial assistance, the first author would like to thank both IIT-Roorkee and HRI-Prayagraj.

	\bibliographystyle{amsplain}

\end{document}